\documentclass{amsart}
\date{July 1, 2019}
\usepackage{graphicx}

\usepackage[dvipsnames]{xcolor}

\title[Maximal surfaces containing entire null line]{
Space-like maximal surfaces containing\\
entire null lines 
in Lorentz-Minkowski 3-space}

\author{S. Akaminte}
\address[Shintaro Akamine]{%
Graduate School of Mathematics, 
Nagoya University, Chikusa-ku, Nagoya 464-8602, Japan.}
\email{s-akamine@math.nagoya-u.ac.jp}

\author{M.~Umehara}
\address[Masaaki Umehara]{%
   Department of Mathematical and Computing Sciences,
   Tokyo Institute of Technology
   2-12-1-W8-34, O-okayama, Meguro-ku,
   Tokyo 152-8552, Japan.
}
\email{umehara@is.titech.ac.jp}

\author{K.~Yamada}
\address[Kotaro Yamada]{%
   Department of Mathematics\\
   Tokyo Institute of Technology\\
   O-okayama, Meguro, Tokyo 152-8551, 
   Japan
}
\email{kotaro@math.titech.ac.jp}

\newcommand{\op}[1]{{\operatorname{#1}}}

\newcommand{\mb}[1]{\vect{#1}}

\newcommand{\vect}[1]{\boldsymbol{#1}}
\newcommand{\R}{\boldsymbol{R}}

\newcommand{\Z}{\boldsymbol{Z}}
\renewcommand{\phi}{\varphi}

\newcommand{\pmt}[1]{{\begin{pmatrix} #1  \end{pmatrix}}}

 \newtheorem{theorem}{Theorem}[section]
 \newtheorem{proposition}[theorem]{Proposition}
 
 \newtheorem{fact}[theorem]{Fact}
 
 \newtheorem{corollary}[theorem]{Corollary}

\theoremstyle{definition}
 \newtheorem{definition}[theorem]{Definition}
\theoremstyle{remark}
 \newtheorem{remark}[theorem]{Remark}
 \newtheorem*{remark*}{Remark}

 \newtheorem*{acknowledgement}{Acknowledgement}

\begin{document}
\maketitle

\begin{abstract}
Consider a surface $S$ immersed in 
the Lorentz-Minkowski 3-space $\R^3_1$.
A complete light-like line in $\R^3_1$ is called
an {\it entire null line} on the surface $S$
in $\R^3_1$ if it lies on $S$ 
and consists of only null points with
respect to the induced metric.  
In this paper, we show the existence of embedded
space-like maximal graphs
containing entire null lines.
If such a graph is defined on a convex domain in $\R^2$,
then it must be a light-like plane 
(cf. Remark \ref{rmk:new}).
Our example is critical in the sense that 
it is defined on a certain non-convex domain.
\end{abstract}

\section{Introduction} \label{sec:1} 

We let $\R^3_1$ be the Lorentz-Minkowski 3-space
of signature $(++-)$.
It is well-known that there are no complete  space-like
zero mean curvature immersions in $\R^3_1$ other than
planes (cf. \cite{CY}). 
It is thus interesting to investigate 
singularities of zero mean curvature surfaces.
In fact, maximal surfaces with singularities
were investigated by several geometers,
see references in \cite{UY1}.
Regarding these works,
the second and third authors introduced
the concept of {\it maxface} in
\cite{UY1}, which covers a large class of
space-like surfaces
with zero mean curvature admitting
singular points. 
For example, the elliptic catenoid 
$$
f_0:=(\sinh v\cos u,\sinh v \sin u,v)\,\, (|u|\le \pi,\,\,v\in \R)
$$
is a typical example of a maxface
 with a cone-like singular point (cf.~Fig.~1,~left).

\begin{figure}[htb]
 \begin{center}
        \includegraphics[height=2.5cm]{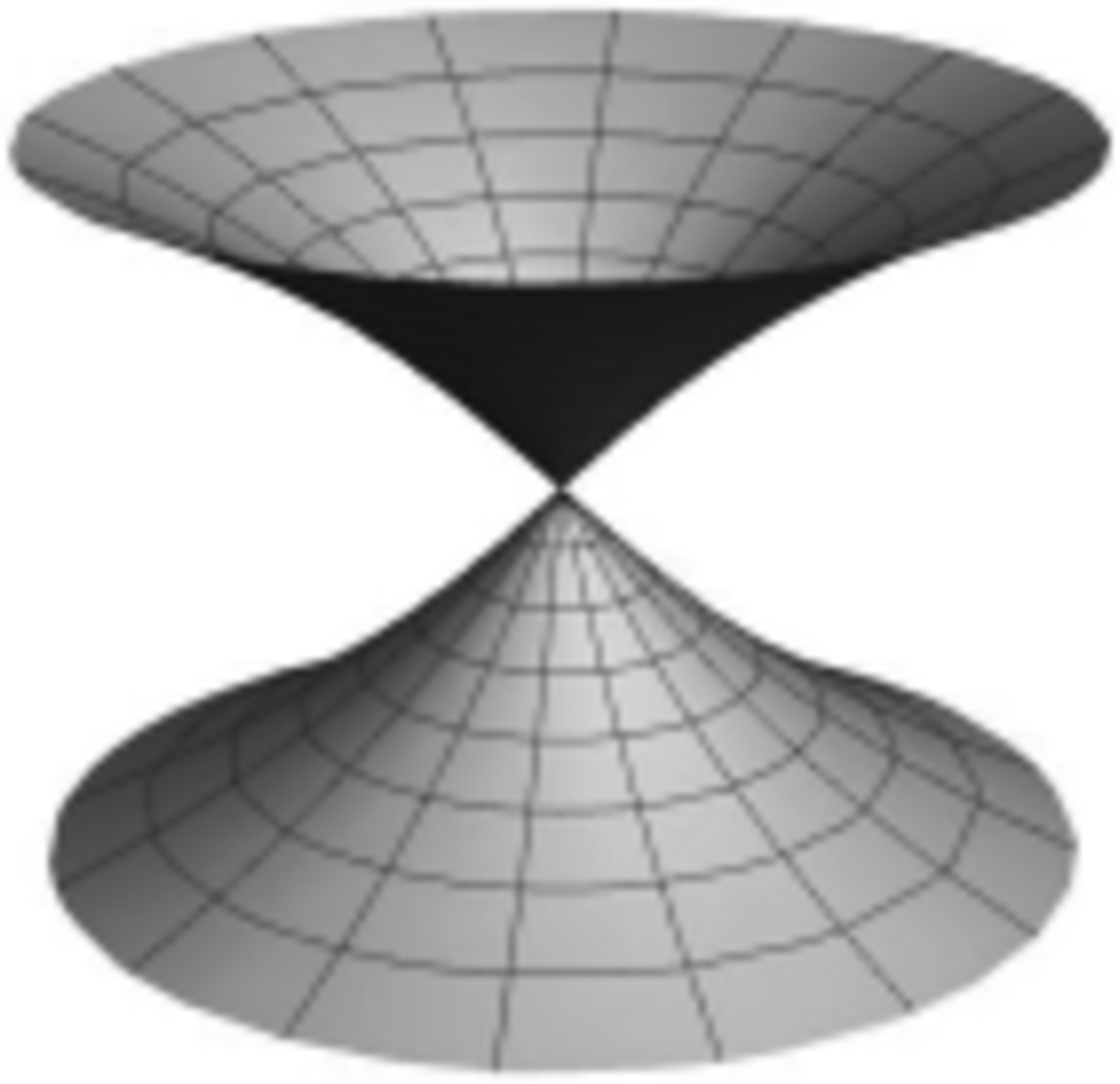}\quad
        \includegraphics[height=2.7cm]{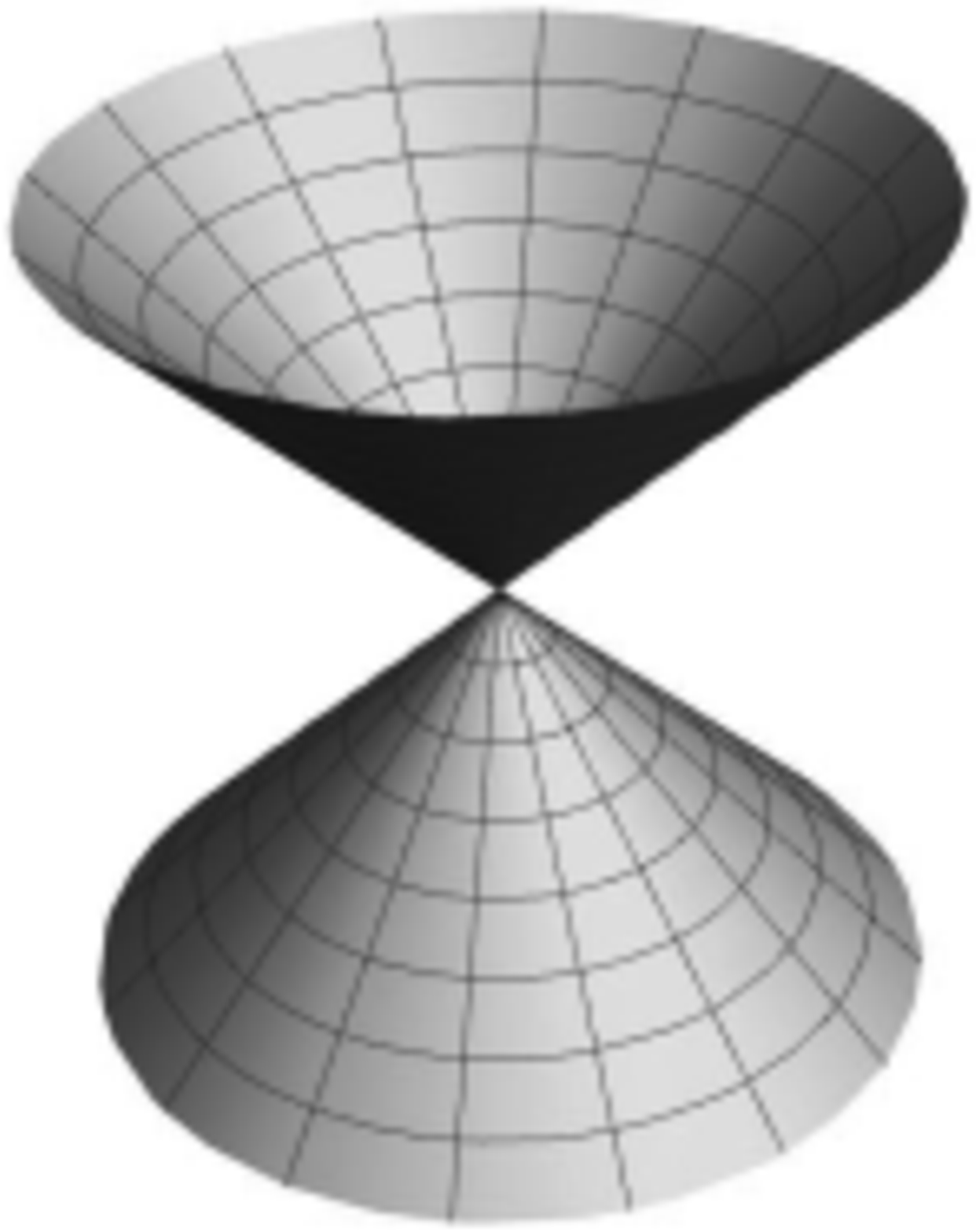}
 \end{center}
\caption{The elliptic catenoid (left) and the light-cone
 (right)}
\label{fig:KM}
\end{figure}

However, if space-like maximal surfaces
have analytic extensions which change their causal type,
then those extensions cannot remain in the class of maxfaces.
In \cite{UY3}, two functions $A_f$ and $B_f$
for immersed hypersurfaces in Lorentzian manifolds
were introduced.
Using them, we  give here a new notion \lq{ZMC-map}\rq\ 
as an appropriate new class for
zero mean curvature surfaces with singularities, 
as follows
(maxfaces are all ZMC-maps but 
the converse is not true):
Let $U$ be a domain in the $uv$-plane, and let 
$f:U\to \R^3_1$ be a smooth map
into the Lorentz-Minkowski 3-space $\R^3_1$.
We set
$$
P:=
\pmt{f_{u}\cdot f_{u}\, &\, f_{u}\cdot f_{v} \\
f_{v}\cdot f_{u}\, &\, f_{v}\cdot f_{v} 
}
$$
and 
$$
B_f:=\det(P),
$$
where $\cdot$ denotes the canonical Lorentzian product of
$\R^3_1$ and $\det(P)$ denotes the determinant of the $2\times 2$
matrix $P$. 
We set
$$
Q:=
\pmt{f_{uu}\cdot \tilde \nu\, &\, f_{uv}\cdot \tilde \nu \\
f_{vu}\cdot \tilde \nu\, &\, f_{vv}\cdot \tilde \nu 
},
$$
where 
$$
\tilde \nu:=\pmt{
1 & 0 & 0 \\
0 & 1 & 0 \\
0 & 0 & -1 
} (f_u\times_E f_v)
$$
and $f_u\times_E f_v$ denotes 
the (Euclidean) vector product of $\R^3$.
Here $\tilde \nu$ gives a (Lorentzian) normal vector 
field of $f$ defined on $U$.
Consider the matrix
$$
W:=\tilde P Q
$$
and set
$$
A_f:=\op{trace}(W), 
$$
where $\tilde P$ is the cofactor matrix of $P$.
We call $f$ a {\it ZMC-map}
(i.e.~{\it zero mean curvature map})
if it is an immersion on an open dense subset of $U$
and $A_f$ vanishes identically.
Since the property that $A_f$ varnishes
is independent of the choice of local coordinates,
ZMC-maps can be defined from an
arbitrarily given $2$-manifold.
Moreover, one can generalize the concept of
ZMC-map for a map into $\R^{n+1}_1$ from an
arbitrarily given $n$-manifold 
using the two functions $A_f$ and $B_f$
given in \cite{UY3}.
We let $f$ be a ZMC-map.
A point $p\in U$ is said to be a {\it space-like point}
(resp. {\it time-like point}) of $f$, 
if $B_f(p)>0$ (resp. $B_f(p)<0$).
A point which is neither space-like nor time-like is
said to be {\it a null point} or {\it a light-like point} of $f$.
(If $p$ is a singular point of $f$, then $B_f(p)=0$.
So $p$ is a null point.)

If $f:U\to \R^3_1$ is a light-like surface, that is,
all points on $U$ are null points, then 
$A_f$ vanishes identically (cf. \cite[Example 4]{Kl}).
So, by our definition, a light-like surface
is also an example of a ZMC-map. 

\medskip
\begin{definition}
A ZMC-map is said to be  {\it maximal type} 
(resp. {\it of mixed type})
if $B_f\ge 0$ but does not vanish identically
(resp.~$B_f$ changes its sign on $U$.)
On the other hand, 
a ZMC-map is said to be  {\it null} or {\it light-like} if
$B_f$ vanishes identically.
A ZMC-map which is neither maximal, 
of mixed type nor light-like
is said to be {\it time-like}.
\end{definition}


\medskip
We now fix $f:U\to \R^3_1$
as a ZMC-map.
A null point $p\in U$ is called  {\it degenerate} if
the exterior derivative of $B_f$ vanishes at $p$.
Consider the light-cone
$$
f_1:=(v\cos u,v \sin u,v)\qquad
(0\le u<2\pi,\,\,v\in \R),
$$
which is a light-like ZMC-map 
consisting only of degenerate null points.

\medskip
\begin{definition}
We say that a ZMC-map $f:U\to \R^3_1$
{\it contains a null line} if
there exists an open interval $I$ and
a smooth curve $\gamma:I\to U$ such that
$\gamma(I)$ consists only of null points
and $f\circ \gamma(I)$
is a subset of a light-like line in $\R^3_1$.
In this case,  $f\circ \gamma(I)$ is called a
{\it null line}.
Moreover, if $f\circ \gamma(I)$ coincides 
with a complete light-like line, we call
it an {\it entire null line}.
\end{definition}

\medskip
Klyachin \cite{Kl} showed the following fact:

\medskip
\begin{fact}[The line theorem for ZMC-surfaces]\label{fact:IL}
{\it Let $f\colon U\to \R^3_1$ be a ZMC-immersion 
such that $o\in U$ is a degenerate 
null point. Then, $f$ contains
a null line passing through $f(o)$.}
\end{fact} 

\begin{figure}[htb]
 \begin{center}
   \begin{tabular}{c@{\hspace{0.3cm}}c@{\hspace{-0.4cm}}c}
        \includegraphics[height=2.5cm]{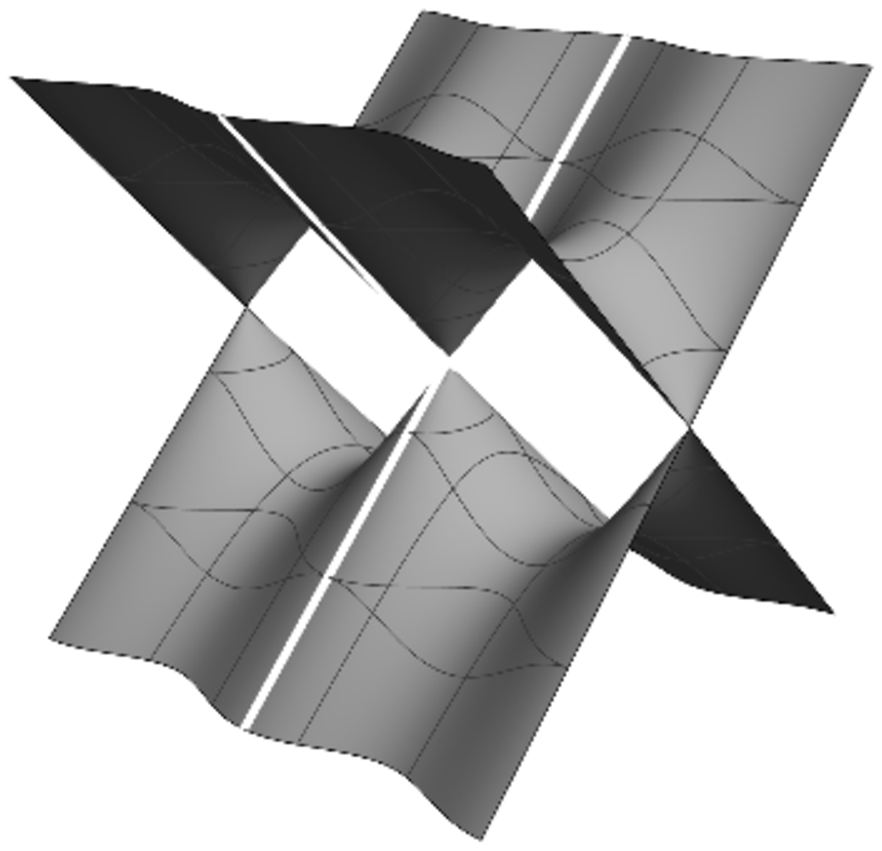} &
        \includegraphics[height=2.5cm]{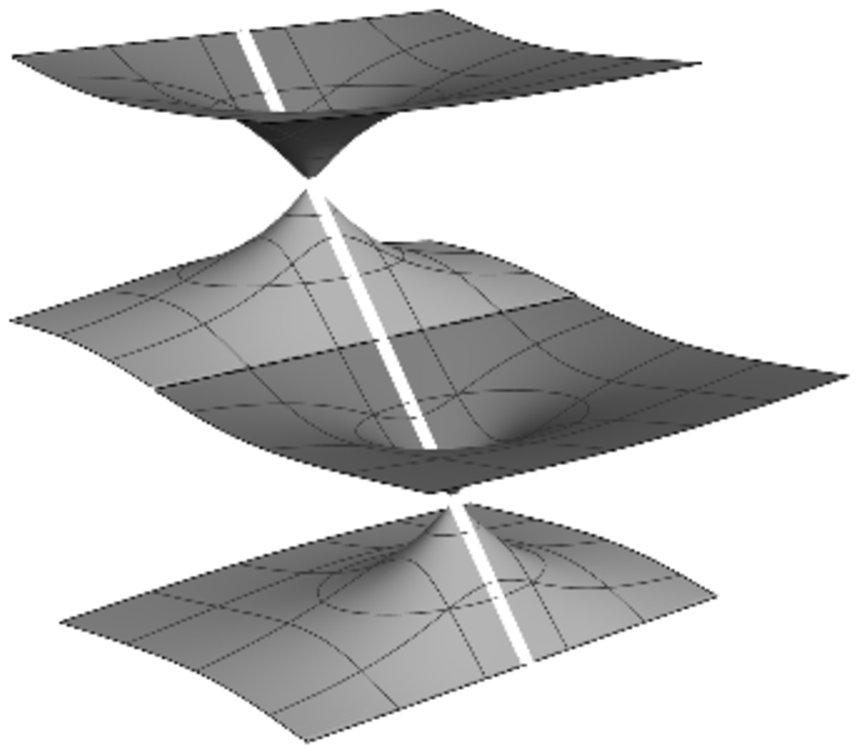} &
    \end{tabular}
\caption{Maximal surfaces with cone-like singular points
lying on null lines, where the white lines indicate null points.}
\label{fig:0}
 \end{center}
\end{figure}

This fact was generalized to a much wider class of surfaces,
including real analytic constant mean curvature surfaces in $\R^3_1$, 
see \cite{UY2,UY3}. 
Although 
there are properly embedded time-like ZMC-surfaces
with an entire null line (see  \cite[Examples 2.2 and 2.3]{AUY}),
each of all examples of ZMC-surfaces with space-like points 
given in \cite{A,A3,CR,OK}
containing an entire null line $L$ 
has at least one cone-like singular point on $L$ (see Fig.~2).

For example, Fig.~2,~left 
is the hyperbolic catenoid (as a ZMC-map)
given in \cite{OK}. (This surface
is called the catenoid of 2nd kind 
in \cite{OK})
and satisfies 
\begin{equation}\label{eq:left}
\sin^2 x + y^2 = t^2,
\end{equation}
where $(x,y,t)$ is the canonical coordinate system of $\R^3_1$, and
$y=\pm t,\,\, x=k\pi$ ($k\in \Z$)
are entire null lines on the surface.

On the other hand,
the maximal ZMC-surface
given in \cite[Theorem 5.3 (1-i)]{A} 
satisfies
\begin{equation}\label{eq:center}
2(y-t)\sin t=(x^2+y^2-2yt+t^2)\cos t
\end{equation}
and $y=t,\,\,x=0$ gives the
entire null line on the surface
(see Fig. 2, right).
By \eqref{eq:left} and
\eqref{eq:center}, these two surfaces
have no self-intersections.

The purpose of this paper is to show the
following:

\medskip
\begin{theorem}\label{thm:main}
{\it There exist embedded maximal ZMC-surfaces 
 $($resp.~embedded ZMC-surfaces 
of mixed type$)$ each of which contains an entire null line.}
\end{theorem}

\medskip
\begin{remark}\label{rmk:main}
This theorem gives the first example of 
maximal ZMC-immersion containing an entire null line.
On the other hand, examples of ZMC-maps 
of mixed type containing an entire null line
were given in \cite[Example 9.3 for $m=1$]{HK} 
using a different approach.  
\end{remark}

\medskip
\begin{corollary}\label{thm:cor}
There exists a family of embedded
maximal ZMC-hypersurfaces 
 $($resp.~embedded ZMC-hypersurfaces 
of mixed type$)$ 
in
Lorentz-Minkowski space $\R^{n+1}_1$
each of which contains an $(n-1)$-dimensional light-like plane.
\end{corollary}

The proof is completely same as in the proof
of \cite[Corollary 1.2]{OJM}.
The strategy to prove 
Theorem \ref{thm:main}
is as follows:
In \cite{OJM}, local existence of
a 1-parameter family $\{f_c\}$
of ZMC-surfaces 
of mixed type such that  each $f_c$ 
contains a null line segment was shown.
By improving the argument there, we
will show in Section 3
that each $f_c$ can be analytically extended
so that it contains an entire null line.
Also, by modifying the estimates in
\cite{OJM}, we will also show the
existence of ZMC-surfaces 
of maximal type each of which contains 
an entire null line.
Unfortunately, our construction is local, and so
the resulting surfaces are not proper. 
So the following question still remains
(this is essentially the same question 
as in \cite[Question 1]{AUY}):

\medskip
{\bf Question.}
{\it Are there properly embedded maximal surfaces
in $\R^3_1$, 
which contain at least one entire null line?}

\medskip
As a partial answer of this question, a Bernstein type theorem
for entire ZMC-graphs
consisting only of space-like or null points
was given in the authors' previous work \cite{AUY}.
As we have mentioned in Remark \ref{rmk:main},
Hashimoto and Kato \cite{HK} recently
gave a new method for constructing  ZMC-maps 
containing null lines, using bi-complex
extensions. The authors expect that this could
be developed to apply to the above question.

\bigskip
\section{ZMC-surfaces with null lines}

We consider a ZMC-immersion $f$ containing the entire
null line  in $\R^3_1$.
By a suitable Lorentzian transformation of the ambient space,
we may assume that this entire null line is given by
$$
L:=\{(0,y,y)\in \R^3_1\,;\, y\in \R\}.
$$
On a neighborhood of a null point, 
there exist a domain $U$ in the $xy$-plane
containing the $y$-axis
and a real analytic function $\psi:U\to \R$
such that $F$ can be expressed as the graph of a function
of the form 
\begin{equation}\label{zm22b}
 \psi(x,y):=y+\frac{\alpha(y)}{2}x^2+\sum_{k=3}^\infty \frac{\beta_k(y)}k x^k,
\end{equation}
where $\alpha$ and $\beta_k$ ($k=3,4,5,\dots$)
are certain real analytic functions on $\R$.
Since 
$$
f(x,y)=(x,y,\psi(x,y)),
$$
we have
\begin{align*}
A_f&=(1-\psi_y^2)\psi_{xx}+2\psi_x\psi_y 
\psi_{xy}+(1-\psi_x^2)\psi_{yy}, \\
B_f&=1-\psi_x^2-\psi_y^2.
\end{align*}
Since $f$  is a ZMC-graph, the function $A:=A_f$ vanishes identically. 
So we have
 \begin{equation}\label{eq:a2}
    0=A_{xx}|_{x=0}=2\alpha\alpha'+\alpha'',
 \end{equation}
 where the prime means the derivative with respect to $y$.
 Hence there exists a constant $\mu$ such that
\begin{equation}\label{eq:c}
\alpha'+\alpha^2+\mu=0.
\end{equation}
When $\mu>0$, then $f$ is space-like,
and when $\mu<0$, then $f$ is time-like.
If $\mu=0$,  the causal type of $f$ cannot be
specified, and the
following four possibilities arise:
\begin{enumerate}
\item[(i)] the graph of $\psi$ is of mixed type, 
\item[(ii)] the graph of $\psi$ consists of space-like points, 
except for the $y$-axis,
\item[(iii)] the graph of $\psi$ consists of time-like points except for the $y$-axis, and
\item[(iv)] the graph of $\psi$ consists only of null points.
\end{enumerate}
The final case (iv) occurs only when
the graph of $\psi$ lies on a light-like 
plane in $\R^3_1$ (cf. \cite{AUY}).

\medskip
\begin{remark}
If we weaken the condition that the image of $f$ contains a
null line segment, an example for the case (i)
was shown in \cite{OJM}, and a general local existence theorem for
such $\psi$ satisfying (i), (ii) and (iii) was shown
in \cite[Prop. 6.7]{UY2}. 
\end{remark}

\medskip
As shown in \cite{CR} and \cite{UY2},
by a homothetic change
\begin{align*}
   &\widetilde{f}(x,y):=(x,y, \tilde{\psi}(x,y))\\
  &\left(\tilde \psi(x,y):= \frac{1}{m}\psi(mx,my),~m>0\right),
\end{align*}
one can normalize  $\mu$ to be $-1$, $0$ or $1$.
In fact, as shown in \cite{CR},
\[
 \alpha^+:=-\tan(y+c)
 \qquad (|c|<\frac{\pi}2)
\]
is a general solution of $\alpha'+\alpha^2+\mu=0$ for $\mu=1$.
If $\mu=0$, then
\[
  \alpha^0_{I}:=0 \quad
  \mbox{and}
  \quad \alpha^0_{I\!I}:= \frac{1}{y+c} 
  \qquad (c\in \R\setminus \{0\})
\]
are the solutions, and
\begin{alignat*}{2}
 \alpha^-_{I}&:=\tanh(y+c)
 \qquad &(c&\in\R),
 \\
 \alpha^-_{I\!I}&:=\coth(y+c)
 \qquad &(c&\in \R\setminus \{0\}),\\
 \alpha^-_{I\!I\!I}&:=\pm 1 
\end{alignat*}
are the solutions for $\mu=-1$.
If 
$$
\alpha=\alpha^+,\,\, \alpha^0_{I\!I},\,\,
\alpha^-_{I\!I}, 
$$
then they are not defined on $\R$,
so these three cases cannot produce
any embedded ZMC-surfaces with entire null lines.
In fact, the maximal surface given
in \eqref{eq:center} is of type $\alpha^+$. 
On the other hand,
the light cone (Fig.~1, right) and
the maximal surface given
in \eqref{eq:left} is of type $\alpha^0_{I\!I}$ on each null line. 
Moreover, the time-like surface
given by the implicit function
\begin{equation}
(y - t + \tanh t)^2 + x^2 = \tanh^2t
\end{equation}
has a null line $y=t,\,\,x=0$ and 
is of type $\alpha^-_{I\!I}$ (see Fig.~3).
So if we seek candidates of
ZMC-immersions containing
entire null lines, then only the 
possibilities are
$$
\alpha=\alpha^0_{I},\,\,
\alpha^-_{I}, \,\, \alpha^-_{I\!I\!I}.
$$
The cases $\alpha=\alpha^-_{I}, \,\, \alpha^-_{I\!I\!I}$ 
give only time-like ZMC-surfaces,
and properly embedded examples
were known (cf. \cite[Examples 2.2 and 2.3]{AUY}).
So only the case $\alpha:=\alpha^0_{I}(\equiv 0)$ 
is remaining.

\begin{figure}[htb]
 \begin{center}
   \begin{tabular}{c@{\hspace{0.3cm}}c@{\hspace{-0.4cm}}c}
        \includegraphics[height=3.5cm]{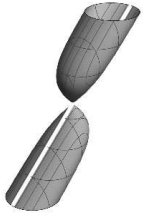} 
    \end{tabular}
\caption{A time-like ZMC-surface without self-intersection
which has a cone-like singular point
lying on a null line.}
\label{fig:3}
 \end{center}
\end{figure}

This special case was discussed in \cite{OJM},
and here we would like to point out  
the method used in \cite{OJM} is sufficient
to show the following assertion:

\medskip
\begin{theorem}\label{MainThm}
There exist real analytic functions $\psi_i:U\to \R$
$(i=1,2,3)$
defined on a domain $U$ in $\R^2$, each of whose
graphs gives a ZMC-embedding containing 
an entire null line of type $\alpha^0_I$,
so that $\psi_1$, $\psi_2$ and $\psi_3$
satisfy the conditions $(i)$, $(ii)$ and $(iii)$,
respectively.
\end{theorem}

\medskip
The existence of $\psi_1$ and $\psi_2$  proves Theorem~\ref{thm:main}
in the introduction.

\medskip
\section{Proof of Theorem \ref{MainThm}}

We first consider the case (i).
For each $c>0$,
a function $\psi$ of a ZMC surface of 
mixed type satisfying $\alpha=0$ was constructed
in \cite{OJM}.
Such a surface can be characterized by the condition
\begin{equation}\label{eq:b3}
\beta_3(y)=3cy,\,\,\, 
\beta_k(0)=\beta'_k(0)=0 \qquad (k\ge 4).
\end{equation}
For an arbitrary $\delta>0$,
there exists a positive constant $C_\delta$
such that $\psi$ is defined on
$$
V_\delta:=(-C_\delta^{-1},C_\delta^{-1})\times (-\delta,\delta).
$$

\begin{remark}\label{rem:ojm}
In \cite[Page 290]{OJM}, 
the inequality
\begin{equation}\label{eq:ojm2}
|\beta_l|\le \frac{3c |y|^{l^*+2}}{(l^*+2)^2}M^{l-3}\le \theta_0 c^l
\end{equation}
was shown, where
$C_\delta:=\delta M$ and $\theta_0:=3(\delta M)^3/c$.
However, there was a typographical error, and
we should correct
\begin{equation}\label{eq:Ct}
C_\delta:=\sqrt{\delta} M, \qquad \theta_0:=\frac{3c}{\sqrt{\delta}M^3}.
\end{equation}
Then \eqref{eq:ojm2} holds, correctly.
\end{remark}

The uniqueness of the solution $\psi$
implies that $\psi$ can be extended on the domain
$
U:=\bigcup_{\delta \geq1} V_\delta.
$ 
In particular, $U$ contains the entirety of the $y$-axis, 
and so $\psi$ gives an example of type (i)
containing the entire null line $L$.

We show the existence of surfaces satisfying (ii)
or (iii) as a modification of the proof of \cite{OJM}.
We will change the initial condition 
and modify the estimates in \cite{OJM} as follows: 
We set
\begin{align}\label{eq:b4}
&\beta_3(y)=0,\quad \beta_4(y)=4cy\qquad (c\ne 0),\\ 
&\beta_k(0)=\beta'_k(0)=0 \qquad (k\ge 5),
\nonumber
\end{align}
as the initial condition.
Then, by \cite[(6.5)]{UY2},  
\begin{align*}
B_f(x,0)&=-2cx^4+\mathcal{O}(x^5)
\end{align*}
holds. 
Hence, if the power series in 
\eqref{zm22b} with the condition \eqref{eq:b4} converges, 
the graph of $\psi$ 
satisfies (ii) (resp. (iii))
if $c<0$ (resp. $c>0$).
The convergence of \eqref{zm22b} can be proved by 
 imitating the argument in \cite{OJM}.
In \cite{OJM}, functions $b_i$ ($i\ge 0$) were used, where 
$$
b_0:=y\quad b_1:=0,\quad
b_2:=\alpha(=0)
$$
and
$
b_k=\beta_k
$
holds for $k\ge 3$.
Since $b_2=b_3=0$ in \cite{OJM}, 
the series  $\{P_k,Q_k,R_k\}$
were produced by the following recursive formula:
We set
$P_j=Q_k=R_l=0$ for $3\le j\le 5$,
$ 3 \le k\le 9$, $3 \le l \le 10$, 
and
\begin{align*}
  P_k & := \sum_{m=4}^{k-2} \frac{2(k-2m+3)}{k-m+2}
\beta_m\beta'_{k-m+2},\,\,\,
k \ge 6, \\
  Q_k & := 
     \sum_{m=4}^{k-6}\sum_{n=4}^{k-m-2}
         \frac{3n-k+m-1}{mn}\beta'_{m}\beta'_{n}\beta_{k-m-n+2}, \\
&\phantom{aaaaaaaaaaaaaaaaaaaaaaaaaaaaaa}k \ge 10, \\
  R_k & :=
     \sum_{m=4}^{k-7}\sum_{n=4}^{k-m-3}
         \frac{\beta_m\beta_n\beta''_{k-m-n+2}}{k-m-n+2},\,\,\,
k \ge 11. 
\end{align*}
Here $P_k,\,Q_k$  and $R_k$ are determined by $\beta_{j}$
($4\le j\le k-2$). 
Thus, the functions $\beta_k$ ($k\ge 5$)
are inductively determined by the 
ordinary differential equation
\begin{equation}
\label{eq:diffeq}
\beta''_k
   = -k (P_k+Q_k-R_k),\,\, \beta_k(0)=\beta'_k(0)=0, 
\end{equation}
where $k \ge 5$.  
The following proposition holds, which 
corresponds to \cite[Proposition 1.3]{OJM}: 

\medskip
\begin{proposition}
\label{prop:goal}
 For each $c\ne 0$ and $\delta>0$, we set
 \begin{equation}\label{eq:M}
   M_\delta:= 3\max\left\{ 144\tau\,|c|\, |\delta|^{3/2},\,
              \sqrt[4]{192 c^2 \tau}
       \right\},
 \end{equation}
 where $\tau$ is a positive constant such that
 \begin{equation*}\label{eq:tau}
     t \int_{t}^{1-t}\frac{du}{u^2(1-u)^2}
 \leq \tau\qquad \left(0<t<\frac{1}{2}\right).
 \end{equation*}
 Then the functions $\beta_l(y)$ $(l\ge 5)$ 
 satisfy the inequalities
 \begin{align}
   |\beta''_l(y)|&\leq |c|\,|y|^{l^*}M^{l-3},\label{eq:b-est-1}\\
   |\beta'_l(y)|&\leq \frac{3|c|\,|y|^{l^*+1}}{l^*+2}M^{l-3},\label{eq:b-est-2}\\
   |\beta_l(y)|&\leq \frac{3|c|\,|y|^{l^*+2}}{(l^*+2)^2}M^{l-3} \label{eq:b-est-3}
 \end{align}
 for any 
$y\in [-\delta,\delta]$,
 where
 \begin{equation*}\label{eq:l-star}
    l^*:=\frac{1}{2}(l-1)-2\qquad (l=5,6,\dots).
 \end{equation*}
\end{proposition}

\begin{proof}
We can prove this by induction on $l\geq 5$.
The functions $\beta_k$ of small indices are determined by the 
recursive formula \eqref{eq:diffeq} as follows:
\begin{equation}\label{eq:list_beta_l}
\begin{aligned}
&\beta_4=4cy,\quad \beta_5=0,\quad
\beta_6=-8 c^2 y^3,\\
&\beta_7=0,\quad
\beta_8=-32c^3y^5.
\end{aligned}
\end{equation}
The estimates \eqref{eq:b-est-1}--\eqref{eq:b-est-3} 
for $l=5$ follow from \eqref{eq:list_beta_l}.
In the case of \cite{OJM},
we set $b_3(=\beta_3)\ne 0$, but
in our present case $b_3=0$ and the first non-trivial
term begins from $b_4(=\beta_4)$.
In particular, sub-terms appearing in $P_k,Q_k,R_k$ 
are fewer than those in \cite{OJM}.
So, \eqref{eq:b-est-1}--\eqref{eq:b-est-3} for $l\geq 6$
follow using the same induction argument as in \cite{OJM}.
Therefore, the 
same estimates as in the proof of 
Proposition 1.9 in \cite{OJM} are entirely valid also in this case.
\end{proof}

By Proposition \ref{prop:goal},
for an arbitrary $\delta \geq 1$,
$\psi$ is well-defined on
$$
V_\delta:=(-C_\delta^{-1},C_\delta^{-1})\times (-\delta,\delta),
$$
where $C_\delta:=\sqrt{\delta} M$ (cf. \eqref{eq:Ct}).
In fact, like as in the case of \cite{OJM},
the inequality
\begin{equation}\label{eq:ojm}
|\beta_l|\le \frac{3|c|\,|y|^{l^*+2}}{(l^*+2)^2}M^{l-3}
\le \theta_0 C_\delta^l
\qquad (l\ge 5)
\end{equation}
can be shown, where
$\theta_0:=3|c|/(\sqrt{\delta} M^3)$.
The uniqueness of the solution $\psi$
implies that $\psi$ can be extended on the domain
\begin{equation}\label{eq:U}
U:=\bigcup_{\delta \geq1} V_\delta.
\end{equation}
In particular, $U$ contains the entirety of the $y$-axis, 
and so $\psi$ gives an example of type (ii)
(resp. (iii)) if $c<0$ (resp. $c>0$)
containing the entire null line $L$.

\begin{remark}\label{rmk:new}
The domain $U$
defined in \eqref{eq:U} is not convex in $\R^2$.
However, this fact is crucial to
construct a maximal graph  containing an entire line.
In fact, as pointed out in \cite[Lemma 1]{AHUY}, 
if there exists a maximal graph defined on
a closed convex domain in $\R^2$ whose image contains
an entire null line, then it must be a light-like plane.
So if there exists an entire graph of mixed type
containing an entire null line $L$, then either
\begin{itemize}
\item the both side of the entire null line are time-like, or
\item 
there exists a sequence of time-like points 
(on the side containing space-like points) 
which is asymptotic to
the line $L$,  
\end{itemize}
as a consequence.
\end{remark}

\begin{acknowledgement}
The authors would like to express their gratitude to
Atsufumi Honda and Wayne Rossman for helpful comments.
\end{acknowledgement}

\end{document}